\documentclass[12pt,leqno,twoside]{amsart}

\usepackage{latexsym,esint, color}
\usepackage{url}

\theoremstyle{plain}

\def\endproof{\hspace*{\fill}\mbox{\ \rule{.1in}{.1in}}\medskip }

\numberwithin{equation}{section}
\numberwithin{figure}{section}

\newtheorem{theorem}{Theorem}[section]
\newtheorem{corollary}[theorem]{Corollary}
\newtheorem{lemma}[theorem]{Lemma}

\newtheorem{definition}[theorem]{Definition}

\newcommand{\R}{\mathbb R} 
\newcommand{\sym}{{\mathrm sym}}
\newcommand{\ds}{\displaystyle} 

\theoremstyle{definition}
\newtheorem{example}[theorem]{Example}
\newtheorem{remark}[theorem]{Remark}

\numberwithin{equation}{section}
\numberwithin{figure}{section}

\def\Xint#1{\mathchoice
{\XXint\displaystyle\textstyle{#1}}%
{\XXint\textstyle\scriptstyle{#1}}%
{\XXint\scriptstyle\scriptscriptstyle{#1}}%
{\XXint\scriptscriptstyle\scriptscriptstyle{#1}}%
\!\int}
\def\XXint#1#2#3{{\setbox0=\hbox{$#1{#2#3}{\int}$ }
\vcenter{\hbox{$#2#3$ }}\kern-.6\wd0}}

\def\dashint{\Xint-}

\begin{document}

\title[Shallow shells of lower order]
{The Monge-Amp\`ere constrained elastic theories \\ of shallow shells} 
\author{Marta Lewicka, L. Mahadevan and Mohammad Reza Pakzad}
\address{Marta Lewicka, University of Pittsburgh, Department of Mathematics, 
301 Thackeray Hall, Pittsburgh, PA 15260}
\address{L. Mahadevan, Harvard University, School of Engineering and
  Applied Sciences, and Department of Physics, 
Cambridge, MA 02138}
\address{Mohammad Reza Pakzad, University of Pittsburgh, Department of Mathematics, 
301 Thackeray Hall, Pittsburgh, PA 15260}
\email{lewicka@pitt.edu, lm@seas.harvard.edu, pakzad@pitt.edu}
 
\begin{abstract}
Motivated by the degree of smoothness of constrained embeddings of
surfaces in $\mathbb{R}^3$,  and by the recent applications to the
elasticity of shallow shells, we rigorously derive the $\Gamma$-limit
of 3-dimensional nonlinear elastic energy of a shallow shell of
thickness $h$,  
where the depth of the shell scales like $h^\alpha$ and the applied
forces scale like $h^{\alpha+2}$, in the limit when $h\to 0$. The main  
analytical ingredients are two independent results: a theorem on
approximation of $W^{2,2}$ solutions of the Monge-Amp\`ere equation  
by smooth solutions, and a theorem on the matching (in other words,
continuation) of second order isometries to exact isometries. 
\end{abstract}

\maketitle

\section{Introduction}\label{intro}

The mathematical theory of elastic shells must account for the deformation of thin elastic surfaces of non-zero curvature and the
associated  energetics.  The subject thus brings together the differential geometry of surfaces with the theory of elasticity appropriately
modified to  account for the small aspect ratio of these slender elastic structures. From a practical engineering perspective, many approximate
theories  have been proposed for the mechanical behavior of elastic shells over the last 150 years \cite{Calladine}. The mathematical foundations for  
these theories has lagged behind but has recently been the focus of much attention \cite{ciar, ciar_book, FJMhier} from two perspectives: (i)  
understanding the rigorous derivation of these theories  teaches us about the limits of their practical applicability, and (ii) it also
sheds light on the smoothness of allowable deformations with implications for the properties of elliptic operators that arise naturally in differential geometry. 
Both these questions can be couched in terms of the behavior of the elastic energy of the shell as
a function of its aspect ratio, i.e.  the ratio of its thickness to its curvature and/or lateral extent. For
a flat sheet with no intrinsic curvature,  the stretching energy  per unit area of a thin sheet is proportional to its thickness, while
the bending energy per unit area is proportional to the cube of its thickness. Thus, as the aspect ratio diminishes, the bending energy vanishes
faster than the stretching or shearing energy. This leads to approximately  isometric deformations of the sheet when it is subject to external
loads, with a concomitant theory known as the F\"oppl-von K\'arm\'an theory \cite{karman} of elastic plates. For shallow shells with an intrinsic
curvature, there is a new small parameter corresponding to the product of  the intrinsic curvature and the thickness so that there are various
possible distinguished limits associated with how small or large this parameter  
is independent of the aspect ratio of the shell. This leads to various theories that penalize non-isometric deformations more or less
depending on the relative magnitude of external loads that cause the shell to deform. 

Over the recent past, we are beginning to get a  good understanding of the limiting models and the types of isometries  
involved in various similar contexts  (e.g. see the review and a conjecture in \cite{lepa} for thin shells), but each separate case
enjoys its own peculiarities related to the the geometry of the shell (hyperbolic, elliptic, degenerate or of mixed type), the order of the
relevant approximate isometries on the shell, the linearity or it lack in the PDEs governing them (which become more complicated as the order  
increases) and their regularity.  For shells of thickness  $h$ and depth $h^\alpha$, subject to applied forces that scale 
with the shell depth as $h^{\alpha+2}$, as the thickness of the shell $h \to 0$, depending on the  
choice of $\alpha$, various limiting theories arise. The regimes $\alpha=1$ and $\alpha> 1$ can be treated in a similar manner 
as discussed in another context in \cite{LMP-newbefore} and are not of interest to us in this paper. 

Here we consider  the case  
when the shell is shallow and the forces weak, but not too weak, i.e. when $0<\alpha <1$, so that nearly isometric deformations might
be expected, but as nonlinear constraints, rather than linear ones as appearing in the $\alpha\ge 1$ regimes. 
In the vanishing thickness limit, using the basic 
methods of $\Gamma$-convergence  which were developed in this context
in \cite{FJMgeo, FJMhier}, we derive a new thin film model   (but see also previous results in \cite{lemapa, LMP-newbefore}).   
Our analysis corresponds to a situation where the second order infinitesimal  isometries on a shallow shell reference configuration are given by the out-of-plane 
displacement $h^\alpha v_0$. In analogy with the results in \cite{FJMhier} for plates with flat geometry $v_0=0$ in the energy scaling  
regime $2<\beta<4$,  we recover a Monge-Amp\`ere constraint $\ds \det \nabla^2 v = \det \nabla^2 v_0$ as a second order isometry constraint 
 on the limiting displacements of Sobolev regularity $W^{2,2}$. This emergence of a \lq \lq linearized curvature"  constraint is natural
 from  a mechanical point of view since the Laplacian of the strain which is proportional to the  Laplacian of the stress for linear stress-strain 
relations characterizes local area changes
 and the isometry condition is softened to one that is effectively locally area-preserving. 

Similarly to previously studied cases \cite{FJMhier, lemopa_convex}, in
order to complete the analysis, we need a result  
on the matching property, i.e.  the continuation of second order
isometries, on shallow shells, to exact isometries (under suitable
regularity assumptions), and another result on
approximating the $W^{2,2}$ solutions of the Monge-Amp\`ere equation by
smooth solutions. Contrary to \cite{lemopa_convex, holepa}, the
constraint here is a fully nonlinear one, 
and one cannot take advantage of the full 
degeneracy of $\det\nabla^2 v = 0$ in \cite[Theorems 7 and 10]{FJMhier} (see also \cite{MuPa2}). 
Nevertheless, we are able to prove the applicable versions of these
results in the case $\det\nabla^2 v_0 \equiv c>0$.  In what follows we start with a 
formulation of the problem (in Section 2) and discuss the various possible
choices of $\alpha$ and the known corresponding limiting theories. 
In Sections 3 and 4, we present our main results, namely, the asymptotic behavior of the elastic energies when 
$h\to 0$, for the case  $0<\alpha<1$, and put forward the
corresponding matching property and density results as the main
technical tools. In Sections 5-6, we present the details of our proofs. Section 7 is
dedicated to the study of not necessarily convex $W^{2,2}$ solutions  
of the Monge-Amp\`ere equation, using some key observations by \v{S}ver\'ak from his unpublished manuscript 
\cite{Sve}. These results are a necessary ingredient of the density result in Section 6.  

\section{Elastic shells of low curvature} 
Let $\Omega\subset\mathbb{R}^2$ be open, bounded and simply connected set. 
For a given out-of-plate displacement 
$v_0\in W^{2,2}(\Omega,\mathbb{R})\cap \mathcal{C}^2(\bar\Omega)$, and
for a given exponent
$\alpha>0$, consider a sequence of surfaces:
\begin{equation*}
S_h=\phi_h(\Omega) \quad \mbox{where } \phi_h(x) = \big(x, h^\alpha v_0(x)\big)
\qquad \forall x=(x_1, x_2)\in\Omega,
\end{equation*}
and the family of thin plates $\Omega^h = \Omega\times (-h/2, h/2)$
and thin shallow shells $(S_h)^h$:
\begin{equation}\label{shallow} 
(S_h)^h = \big\{\tilde \phi_h(x, x_3); ~ x\in\Omega, ~ x_3\in(-h/2,
h/2)\big\} \qquad \forall 0<h\ll 1.
\end{equation}
Above, the Kirchhoff-Love extension $\tilde\phi_h:\Omega^h \rightarrow \mathbb{R}^3$ of
the parametrization $\phi_h$, is given by the formula:
\begin{equation}\label{kl}
\tilde\phi_h(x, x_3) = \phi_h(x) + x_3\vec n^h(x)
\qquad\forall (x, x_3)\in\Omega^h,
\end{equation}
while the vector $\vec n^h(x)$ is the unit normal to $S_h$ at $\phi_h(x)$:
$$\vec n^h(x) = 
\frac{\partial_1\phi_h(x) \times \partial_1\phi_h(x)}
{|\partial_1\phi_h(x) \times \partial_1\phi_h(x)|}
= \frac{1}{\sqrt{1+h^{2\alpha} |\nabla v_0|^2}} \big(-h^\alpha \partial_1v_0(x),
-h^\alpha \partial_2v_0(x), 1\big) \qquad\forall x\in\Omega. $$

\medskip
 
The thickness averaged elastic energy of a deformation $u^h$ of
$(S_h)^h$ is now given by:
\begin{equation}\label{IhW}
I^h(u^h)  = \frac{1}{h}\int_{(S_h)^h} W(\nabla u^h)
\qquad \forall u^h\in W^{1,2}((S_h)^h,\mathbb{R}^3).
\end{equation}
The energy density $W:\mathbb{R}^{3\times 3}\longrightarrow
\bar{\mathbb{R}}_{+}$ above, in addition to being $\mathcal{C}^2$ regular in a neighborhood of $SO(3)$,
is assumed to satisfy the normalization, frame indifference and non\-de\-ge\-ne\-racy conditions:
\begin{equation}\label{cond}
\begin{split}
\exists c>0\quad \forall F\in \mathbb{R}^{3\times 3} \quad
\forall R\in SO(3) \qquad
&W(R) = 0, \quad W(RF) = W(F),\\
&W(F)\geq c~ \mathrm{dist}^2(F, SO(3)).
\end{split}
\end{equation}
where $F=\nabla u$, is typically the deformation gradient associated with a mapping $u$.
The following quadratic forms, generated by $W$, will be relevant in the subsequent analysis:
\begin{equation}\label{quad}
\mathcal{Q}_3(F) = D^2 W(\mbox{Id})(F,F), \qquad
\mathcal{Q}_2(F_{tan}) = \min\{\mathcal{Q}_3(\tilde F); ~~ \tilde
F\in\mathbb{R}^{3\times 3}, ~~ (\tilde F - F)_{tan} = 0\}.
\end{equation}
The form $\mathcal{Q}_3$ is defined for  all $F\in\mathbb{R}^{3\times 3}$, 
while $\mathcal{Q}_2$ is defined on the $2\times 2$ principal 
minors $F_{tan}$ of such matrices.
By (\ref{cond}), both forms $\mathcal{Q}_3$ and all $\mathcal{Q}_2$ are nonnegative
definite and depend only on the  symmetric parts of their arguments.

\medskip

Let $f^h\in L^2((S_h)^h, \R^3)$ be a family of loads applied to
the elastic shells under consideration. The total energy is then:
\begin{equation}\label{Jh}
J^h(u^h)  = \frac{1}{h}\int_{(S_h)^h} W(\nabla u^h) - \frac 1h \int_{(S_h)^h} f^h u^h 
\qquad \forall u^h\in W^{1,2}((S_h)^h,\mathbb{R}^3). 
\end{equation} 
In what follows, we will make the simplifying assumptions:
\begin{equation}\label{forces} 
f^h = (0,0, h^{\alpha'}  f \circ \tilde \phi_h^{-1})^T
\end{equation} 
for some $f\in L^2(\Omega)$, normalizes so that: 
\begin{equation}\label{force-properties}
\ds \int_\Omega f = 0\quad \mbox{ and }\quad \int_{\Omega} x f(x)~\mbox{d}x =0.
\end{equation} 

Heuristically, stronger forces ($\alpha' <\alpha+2)$ deform the shallow
shell beyond the reference shape, while 
weaker forces ($\alpha'>\alpha+2$) leave it undeformed, with an
asymptotic behavior of displacements of lower order  
similar to that of a plate. The remaining case where the forces are
tuned with the curvature of the mid-surface (shallowness), 
is given by the scaling regime $\alpha'= \alpha+2$. 

\medskip

When  $\alpha\geq 1$, by a simple change of variables, we see that:
\begin{equation*}
\begin{split}
J^h(u^h) = \frac{1}{h}\int_{\Omega^h} W\Big((\nabla
v^h) (b^h)^{-1}\Big) \det b^h - \frac 1h  \int_{\Omega^h} h^{\alpha_i}
fv^h_3 \det b^h,
\end{split}
\end{equation*}
with $v^h=u^h\circ\tilde\phi_h\in W^{1,2}(\Omega^h,\mathbb{R}^3)$
and $b^h = \nabla \tilde\phi_h$. Note that
by the polar decomposition of positive definite matrices:
$ b^h = R(x,x_3) a^h $ for some $R(x,x_3)\in SO(3)$ and the symmetric tensor $a^h=\sqrt{(b^h)^T b^h}$.
Therefore, for the isotropic energy $W$ i.e. when:
\begin{equation}\label{iso}
W(FR) = W(F) \qquad \forall 
F\in\mathbb{R}^{3\times 3} \quad \forall R\in SO(3)
\end{equation}
one obtains:
\begin{equation}\label{changev}
\begin{split}
I^h(u^h) & = \frac{1}{h}\int_{\Omega^h} W\Big((\nabla
v^h) (a^h)^{-1}R(x, x_3)^{-1}\Big) \det b^h~\mbox{d}(x, x_3) \\ & 
=  \frac{1}{h}\int_{\Omega^h} W\Big((\nabla v^h) (a^h)^{-1}\Big) (1+
\mathcal{O}(h))~\mbox{d}(x, x_3),
\end{split}
\end{equation}
hence reducing the problem to studying deformations of the
flat plate $\Omega^h$ relative to the prestrain tensor $a^h$, see \cite{lemapa, LMP-newbefore} for 
a discussion of this topic.

\medskip

When $\alpha =1$, we derived in \cite{LMP-newbefore}
the $\Gamma$-limit $\mathcal{J}_{v_0}$ of the scaled energies
$\frac{1}{h^4} I^h$. Namely, we showed that the energy of the almost 
minimizing deformations scales like: 
$$\inf J^h\sim h^4,$$ 
and that the $\Gamma$-limit (in the general, possibly non-isotropic case) is given by: 
\begin{equation}\label{vonKnew1}
\begin{split}
\mathcal{J}_{v_0}(w,v)= 
\frac{1}{2} 
\int_\Omega \mathcal{Q}_2&\left(\mathrm{sym }\nabla w 
+\frac{1}{2}\nabla v\otimes \nabla v - \frac{1}{2}\nabla
v_0\otimes\nabla v_0 \right )\\
&\qquad\qquad
 + \frac{1}{24} \int_\Omega \mathcal{Q}_2\Big(\nabla^2 v - \nabla^2v_0\Big) - \int_{\Omega} fv,
\end{split}
\end{equation}
which is the von K\'arm\'an-type functional defined for all out-of-plane displacements $v\in
W^{2,2}(\Omega,\mathbb{R})$ and the in-plane displacements $w\in
W^{1,2}(\Omega, \mathbb{R}^2)$.
In analogy with the theory for flat plates with 
\cite{FJMhier}, due to the choice of scaling in (\ref{forces}), the limit energy is
composed of two terms,  corresponding to stretching and bending. 

In the isotropic case (\ref{iso}), the Euler-Lagrange equations of
$\mathcal{J}_{v_0}$ are:
\begin{equation}\label{new_Karman}
\left \{ \begin{split}
\Delta^2\Phi & = -S(\det \nabla^2 v -\det\nabla^2 v_0)\\
B(\Delta^2v-\Delta^2 v_0) &= [v,\Phi] + f
\end{split}\right.
\end{equation} 
where $S$ is the Young modulus, $B$  the bending stiffness,
$\nu$  the Poisson ratio (given in terms of  the Lam\'e constants
$\mu$ and $\lambda$). 
A more involved version of the system (\ref{new_Karman}) incorporating prestrain was first introduced in
\cite{Maha2} using a thermoelastic analogy to growth, as a
mathematical model of blooming activated by differential lateral
growth from an initial non-zero transverse displacement field $v_0$ (See also \cite{Maha, BenAmar}). 
See \cite{LMP-newbefore} for the rigorous derivation of that model. 

\medskip

When $\alpha>1$,  the bending energy takes over the stretching and hence the limiting theory is a variant of the
linear elasticity as discussed in \cite{FJMhier}, yielding the Euler-Lagrange equations:
\begin{equation*}\label{old_Karman}
B(\Delta^2v - \Delta^2 v_0) =  f.
\end{equation*} 

\medskip
We now turn to the case of interest treated in this paper, namely $0<\alpha <1$.

\section{The main results}\label{secres}

Our first main result regards the matching of isometries on convex weakly shallow shells.


\medskip

\begin{theorem}\label{th8.1}
Assume that $\Omega$ is simply connected and let
$v_0\in\mathcal{C}^{2,\beta}(\bar\Omega,\mathbb{R})$ with 
$\det\nabla^2v_0>c_0>0$. Let
$v\in\mathcal{C}^{2,\beta}(\bar\Omega,\mathbb{R})$ satisfy:
\begin{equation}\label{constraint}
\det\nabla^2 v = \det\nabla^2 v_0 \quad \mbox{ in } \Omega.
\end{equation} 
Then there exists a
sequence $w_h\in\mathcal{C}^{2,\beta}(\bar\Omega,\mathbb{R}^3)$ such
that:
\begin{equation}\label{metric}
\forall h>0 \quad \nabla(\mathrm{id} + hve_3 + h^2w_h)^T
\nabla(\mathrm{id} + hve_3 + h^2w_h) = \nabla(\mathrm{id} + hv_0e_3)^T
\nabla(\mathrm{id} + hv_0e_3) 
\end{equation}
and $\sup \|w_h\|_{\mathcal{C}^{2,\beta}}<+\infty$.
\end{theorem}

We will explain the meaning of the ``matching property'' and put our
analysis in a broader context in Section \ref{match_explain}, while we
give its proof in Section \ref{sec5}.

\medskip

Our next main result concerns the density of regular solutions to the
elliptic 2d Monge-Amp\`ere equation.

\begin{theorem}\label{thm-density}
Let $\Omega$ be open, bounded, connected and star-shaped with respect to an
interior ball $B\subset \Omega$. For a fixed constant $c_0>0$, define:
$$ {\mathcal A}:= \big\{ u\in W^{2,2}(\Omega); ~~ \det\nabla^2 u=c_0
    \mbox{ a.e. in } \Omega \big\}. $$ 
Then $\mathcal A \cap C^\infty( \bar \Omega)$ is dense in ${\mathcal A}$ with 
respect to the strong $W^{2,2}$ norm.
\end{theorem}

We prove Theorem \ref{thm-density} by showing first the interior regularity of solutions
and then by applying a simple scaling argument.
The main difficulty to overcome is the absence of convexity assumptions on the 
$W^{2,2}$ solutions of the Monge-Amp\`ere equation in our context. We 
first establish the convexity of elements of ${\mathcal A}$ in a broader sense,
combining some key observations by \v{S}ver\'ak from his unpublished manuscript 
\cite{Sve}, with a theorem due to Iwaniec and \v{S}ver\'ak in
\cite{IwanSve} (Theorem \ref{thm-IS}) on deformations with integrable dilatation in dimension $2$.  
Some of the results in \cite{Sve} are known within the community,
but for completeness we write the proofs in full detail in Sections 6
and 7.

\medskip

Our remaining results concern the variational limit of the elastic
energy functionals (\ref{Jh}). We first state the following lemma, whose proof is similar to  
\cite[Theorem 2-i]{FJMhier} and \cite[Theorem 2.5]{lemopa7}, and hence
it is omitted for brevity of presentation.

\begin{lemma}\label{energy-scaling}
Assume (\ref{forces}) and (\ref{force-properties}). Then:
\begin{itemize}
\item[(i)] For every small $h>0$ one has:
$$0\geq \inf\left\{\frac{1}{h^{2\alpha+2}} J^h(u^h); ~~ u^h\in W^{1,2}((S_h)^h, \mathbb{R}^3)
\right\}\geq -C.$$
\item[(ii)]  If $u^h\in W^{1,2}((S_h)^h, \mathbb{R}^3)$ is a minimizing sequence
of $\frac{1}{h^{2\alpha+2}} J^h$, that is:
\begin{equation}\label{min_seq}
\lim_{h\to 0} \left(\frac{1}{h^{2\alpha+2}}J^h(y^h) - \inf\frac{1}{h^{2\alpha+2}}J^h \right)=0,
\end{equation}
then  $\frac 1{h^{2\alpha+2}} I^h(u^h)$ is bounded. 
\end{itemize}
\end{lemma}

We now note, by a straightforward calculation, that $a^h$ in
(\ref{changev}), pertaining to the isotropic case (\ref{iso}),
becomes:
$$a^h = \mbox{Id} + \frac{1}{2} h^{2\alpha} (\nabla v_0\otimes \nabla
v_0)^* - h^\alpha x_3 (\nabla^2v_0)^* + o(h^{2\alpha}) + x_3 o(h^\alpha)$$
where the uniform quantities in $o(h^{2\alpha})$, $o(h^\alpha)$ are
independent of $x_3$. Following the proof of Theorem 1.3 in
\cite{lemapa}, we actually obtain in the general (possibly non-isotropic) case:

\begin{theorem}\label{cosik}
Assume that $u^h\in W^{1,2}((S_h)^h,\mathbb{R}^3)$ satisfies
$I^h(u^h)\leq Ch^{2\alpha+2}$, where $I^h$ is given in (\ref{IhW}) and
$0<\alpha<1$. Then there exists rotations $\bar R^h\in SO(3)$ and
translations $c^h\in\mathbb{R}^3$ such that  for the normalized
deformations:
\begin{equation}\label{resc}
y^h(x,t) = (\bar R^h)^T (u^h\circ \tilde \phi_h)(x, ht) - c^h : \Omega^1\longrightarrow\mathbb{R}^3
\end{equation}
defined by means of (\ref{kl}) on the common domain
$\Omega^1=\Omega\times (-1/2, 1/2)$ the following holds:
\begin{itemize}
\item[(i)] $y^h(x,t)$ converge in $W^{1,2}(\Omega^1,\mathbb{R}^3)$ to
  $x$.
\item[(ii)] the scaled displacements $V^h(x) = h^{-\alpha}
  \int_{-1/2}^{1/2}y^h(x,t) - x ~\mbox{d}t$ converge (up to a
  subsequence) in $W^{1,2}(\Omega,\mathbb{R}^3)$ to 
$(0,0,v)^T$ where $v\in W^{2,2}(\Omega, \mathbb{R})$ and:
\begin{equation}\label{constr}
\mathrm{det} \nabla^2 v = \mathrm{det}\nabla^2 v_0.
\end{equation}
\item[(iii)] Moreover: $\liminf_{h\to 0} h^{-(2\alpha+2)} J^h(u^h) \geq
  \mathcal{J}_{v_0} (v)$ where:
\begin{equation}\label{lin_new}
\mathcal{J}_{v_0}(v)= 
 \frac{1}{24} \int_\Omega \mathcal{Q}_2\Big(\nabla^2 v - \nabla^2v_0\Big) - \int_\Omega fv.
\end{equation}
\end{itemize}
\end{theorem}
The constraint (\ref{constr}) in the assertion (ii) follows by observing that
$h^{-\alpha} \mbox{sym}\nabla V^h$ converges in $L^2(\Omega)$ to
$F=\frac{1}{2} (\nabla v_0\otimes\nabla v_0 - \nabla v\otimes \nabla v)$, 
and hence $F=\mbox{sym}\nabla w$ for some $w:\Omega\rightarrow \mathbb{R}^2$.
Consequently $\mbox{curl}^T\mbox{curl} F=0$, which implies (\ref{constr}). 

\medskip

We conjecture that the linearized Kirchhoff-like energy (\ref{lin_new}) (\ref{constr}) is also 
the exact $\Gamma$-limit of the rescaled energies $h^{-(2\alpha+2)} J^h(u^h)$ when $\Omega$ is 
simply connected. Note that this result is established in \cite{FJMhier} 
for the degenerate case $v_0\equiv 0$ and hence our model 
can be considered as a generalization of the linearized Kirchhoff 
model discussed in \cite{FJMhier}. If $\Omega$ is not simply connected, one must replace 
(\ref{constr}) with a more general variant which states that $\nabla
v_0\otimes\nabla v_0 - \nabla v\otimes \nabla v$ is  a symmetric gradient.   

\medskip

We are able to prove our conjecture in the specific case when $\det\nabla^2 v_0$ is constant.
Its proof relies on Theorems \ref{th8.1} and \ref{thm-density} and it
will be given in Section \ref{last}.
Namely, we have:

\begin{theorem}\label{upbd}
Assume that $\Omega$ is open, bounded, connected and star-shaped with
respect to an interior ball $B\subset\Omega$. Assume that:
$$\det\nabla^2 v_0 \equiv c_0>0 \quad \mbox{ in } \Omega.$$
Fix $\alpha\in (0,1)$. Then, for every $v\in W^{2,2}(\Omega,\mathbb{R})$ with 
$\nabla^2 v = \nabla^2 v_0$, there exists a sequence of deformations $u^h\in
W^{1,2}((S_h)^h,\mathbb{R}^3)$ such that:
\begin{itemize}
\item[(i)] The rescaled sequence $y^h(x,t) = u^h(x+h^\alpha v_0(x)e_3 + ht\vec
  n^h(x))$ converges in $W^{1,2}(\Omega^1,\mathbb{R}^3)$ to $x$.
\item[(ii)] The scaled displacements $V^h$ as in (ii) Theorem \ref{cosik}
converge in $W^{1,2}$ to $(0,0,v)$.
\item[(iii)] $\lim_{h\to 0} h^{-(2\alpha+2)}J^h(u^h) = \mathcal{J}_{v_0}(v).$
\end{itemize}
\end{theorem}

\begin{remark}
We expect that the $\mathcal{C}^{2,\beta}$ scalar fields $v$ are dense in the set of the $W^{2,2}$ 
fields with any prescribed, strictly positive but not necessarily
constant $\det\nabla^2$ of $\mathcal{C}^{0, \beta}$ regularity. 
With such a result at hand, it would follow that for all convex shells of sufficient regularity, 
the linearized Kirchhoff-type energy 
(\ref{lin_new}) is the rigorous variational limit on weakly shallow shells, in the 
same spirit as the matching and  density of first order isometries on 
convex shells \cite{lemopa_convex} 
resulted in that the only small slope theory for an elastic convex shell is the linear theory. The latter problems, 
when posed for surfaces of arbitrary geometry, are more difficult. One could hope to prove similar results 
for strictly hyperbolic surfaces $S$. In the general case, however,
such problems reduce to the study of nonlinear PDEs of 
mixed types for which not so many suitable methods are at hand.
\end{remark} 

\section{Remarks on the matching property in Theorem \ref{th8.1}.}\label{match_explain}

\begin{remark}\label{matching}
In \cite{lepa}, the authors put forward a conjecture regarding
existence of infinitely many small slope shell theories (with no
prestrain) each valid for
a corresponding range of energy scalings. This conjecture is based on
formal asymptotic expansions and it is in accordance with the
previously obtained results for plates and shells. It predicts the
form of the 2-dimensional limit energy functional, and identifies the
space of admissible deformations as infinitesimal isometries of a
given integer order $N > 0$ determined by the magnitude of the elastic
energy. Hence, the influence of shell’s geometry on its qualitative
response to an external force, i.e. the shell’s rigidity, is reflected
in a hierarchy of functional spaces of isometries (and infinitesimal
isometries) arising as constraints of the derived theories.

In certain cases, a given $N$th order infinitesimal isometry can be
modified by higher order corrections to yield an infinitesimal
isometry of order $M > N$, a property to which we refer to by ``matching
property of infinitesimal isometries''. This feature, combined with
certain density results for spaces of isometries, cause the theories
corresponding to orders of infinitesimal isometries between $N$ and $M$,
to collapse all into one and the same theory. The examples of such
behavior are observed for plates \cite{FJMhier}, where any second order
infinitesimal isometry can be matched to an exact isometry ($M =
\infty$), for convex shells \cite{lemopa_convex}, where any first order infinitesimal
isometry satisfies the same property, and for non-flat developable  surfaces \cite{schmidt, holepa}
where first order isometries can be matched to higher order isometries (see also \cite{hornung-latest}). 
The effects of these geometric properties on the elasticity of thin films are drastic. A plate whose boundary is at least 
partially free possesses three types of small-slope theories: the linear theory, the
von K\'arm\'an theory and the linearized Kirchhoff theory, whereas
the only small slope theory for a convex shell with free boundary is the linear theory
\cite{lemopa_convex}: a convex shell transitions directly from the linear regime to
the fully nonlinear bending one if the applied forces are adequately
increased. In other words, while the von K\'arm\'an theory describes the
buckling of thin plates at a body force magnitude of order thickness-cubed, 
the equivalent, variationally correct theory for buckling of elliptic shells is the purely 
nonlinear bending theory which comes only into effect when the body forces reach to a magnitude 
of order thickness-squared. 
\end{remark}

\begin{remark}
Writing $w_h = w_{h, tan} + w_h^3e_3$ where $w_{h, tan}(x)
\in\mathbb{R}^2$ and $w_h^3\in\mathbb{R}$, equation (\ref{metric})
becomes:
\begin{equation}\label{metric2}
\begin{split}
\mathrm{Id} &+ h^2(2\mathrm{sym}\nabla w_{h,tan} + \nabla v\otimes\nabla v) +
2h^3\mathrm{sym}(\nabla v\otimes \nabla w_h^3) \\ & + h^4 ((\nabla
w_{h,tan})^T\nabla w_{h,tan} + \nabla w_h^3\otimes \nabla w_h^3)
= \mathrm{Id} + h^2\nabla v_0 \otimes \nabla v_0.
\end{split}
\end{equation}
Recall that (\ref{constraint}) is equivalent to
$\mbox{curl}^T\mbox{curl} (\nabla v\otimes\nabla v - \nabla v_0\otimes
\nabla v_0) = 0$, and hence to: $\nabla v\otimes\nabla v - \nabla v_0\otimes
\nabla v_0 = \mbox{sym}\nabla w$   for some
$w\in\mathcal{C}^{2,\beta}(\bar\Omega,\mathbb{R}^2)$ since $\Omega$ is simply connected.
Hence the constraint (\ref{constraint}) is necessary and sufficient
for matching the lowest order ($h^2$) terms in (\ref{metric2}). 

Our result states that actually it is possible to perturb $w$ by an
equibounded 3d displacement $w_h-w$ so that the full equality
(\ref{metric2}) holds. A natural way for proving this is by implicit
function theorem. Indeed, this is how we proceed, and the ellipticity
assumption $\det\nabla^2v_0>0$ is precisely a sufficient condition for
the invertibility of the implicit derivative $\mathcal{L}(p):
\mathcal{C}^{2,\beta}_0(\bar\Omega,\mathbb{R})\longrightarrow
\mathcal{C}^{0,\beta}(\bar\Omega,\mathbb{R})$, $\mathcal{L}(p) =
-\mbox{cof}\nabla^2 v : \nabla^2p$
where $p$ is the variation in $w_h^3$.
An extra argument for the uniform boundedness of
$w_{h,tan}$ in $\mathcal{C}^{2,\beta}$ concludes the proof.
\end{remark}

\begin{remark}
The condition (\ref{metric}) means that each deformation
$u^h:S_h\rightarrow \mathbb{R}^3$ of a surface
$S_h=\{x+hv_0(x)e_3; ~~ x\in\Omega\}$, given by $u^h(x+hv_0(x)e_3) = x
+hv(x)e_3 + h^2w_h(x)$ is an isometry of $S_h$. In other words, the
pull-back metrics from the Euclidean metric of $S_h$ and of $u^h(S_h)
= \{x+hv(x)e_3 + h^2w_h(x); ~~x\in \Omega\}$ coincide.  Hence 
Theorem \ref{th8.1} asserts that if two convex out-of-plane 
displacements of first order have the same determinants of Hessians,
then they can be matched by a family of equibounded higher order displacements
(the fields $w_h$) to be isometrically equivalent. When the parameter
$h$ is replaced by $h^\alpha$,   
this result will be used  in our context of shallow shells.
For other results  concerning matching of isometries see \cite[Theorem 7]{FJMhier},
\cite[Theorem 1.1]{lemopa_convex}, \cite[Theorem 3.1]{holepa}  
(which is comparable with \cite[Lemma 3.3]{schmidt} and the remark
which follows therein) and \cite[Theorem 4.1]{hornung-latest}.   
\end{remark}

\section{The matching property on convex shallow shells:
a proof of Theorem \ref{th8.1}.}\label{sec5}

By a direct calculation, (\ref{metric}) is equivalent to:
\begin{equation}\label{metric3}
\nabla(\mbox{id} + h^2w_{h,tan})^T \nabla(\mbox{id} + h^2w_{h,tan}) = 
\mbox{Id} + h^2\nabla v_0\otimes\nabla v_0 - h^2(\nabla v + \nabla
z_h) \otimes (\nabla v + \nabla z_h), 
\end{equation}
where $w_{h,tan}\in\mathcal{C}^{2,\beta}(\bar\Omega,\mathbb{R}^2)$ and
$z_h = hw_h^3\in\mathcal{C}^{2,\beta}(\bar\Omega,\mathbb{R})$ so that
$w_h = w_{h,tan} + w_h^3e_3$ is the required correction in
(\ref{metric}).

\medskip

{\bf 1.} We shall first find the formula for the Gaussian curvature of the 2d
metric in the right hand side of (\ref{metric3}):
\begin{equation}\label{8.4}
g_h(z_h) = \mbox{Id} + h^2\nabla v_0\otimes\nabla v_0 - h^2(\nabla v + \nabla
z_h) \otimes (\nabla v + \nabla z_h).
\end{equation}

\begin{lemma}\label{lem8.2}
Let $v_0, v\in \mathcal{C}^{2,\beta}(\bar\Omega,\mathbb{R})$ and consider
the $\mathcal{C}^{1,\beta}$ regular metrics on $\Omega$ of the type:
$$g = [g_{ij}]_{i,j=1,2}= \mathrm{Id} + h^2(\nabla v_0\otimes\nabla v_0
- \nabla v_1\otimes\nabla v_1).$$
Then, for any $h>0$ small, the Gaussian curvature $\kappa(g)$ of $g$
is $\mathcal{C}^{0,\beta}$ regular and it is given by the formula:
\begin{equation}\label{curv}
\kappa(g) = h^2\left[\frac{\mathrm{det}(\nabla^2 v_0 -
  [\Gamma_{ij}^k\partial_kv_0]_{ij})}{\big(1-h^2(g^{ij}\partial_iv_0 \partial_jv_0)\big)^2\mathrm{det}g}
-\frac{\big(1-h^2(g^{ij}\partial_iv_0 \partial_jv_0)\big)^2}{\big(1-h^2|\nabla
  v_1|^2\big)^2}\mathrm{det}\nabla^2 v_1\right],
\end{equation}
where the Christoffel symbols of $g$, the inverse of $g$, and its
determinant are:
\begin{equation}\label{chri}
\Gamma_{ij}^k = \frac{1}{2}g^{kl}\left(\partial_jg_{il} + \partial_i
  g_{jl} - \partial_j g_{ij}\right),
\end{equation}
\begin{equation}\label{inv}
g^{-1} = [g^{ij}] = \frac{1}{\mathrm{det}[g_{ij}]} \mathrm{cof}[g_{ij}],
\end{equation}
\begin{equation*}
\det g = 1-h^4 |(\nabla v_0)^\perp\cdot\nabla v_1|^2 + h^2
(|\nabla v_0|^2 - |\nabla v_1|^2).
\end{equation*}
\end{lemma}
\begin{proof}
Assume first that $v_0$ and $v_1$ are in fact smooth.
By Lemma 2.1.2 in \cite{hanhong}, we have:
\begin{equation*}
\begin{split}
\kappa\big(\mbox{Id} - h^2\nabla v_1\otimes \nabla v_1\big) &=
-h^2\frac{\mbox{det}\nabla^2 v_1}{\big(1-h^2|\nabla v_1|^2\big)^2}\\
\kappa\big(g - h^2\nabla v_0\otimes\nabla v_0\big) & =
\frac{1}{\big(1-h^2(g^{ij}\partial_iv_0 \partial_jv_0)\big)^2}
\left[ \kappa(g) - \frac{h^2 \mathrm{det}(\nabla^2 v_0 -
  [\Gamma_{ij}^k\partial_kv_0]_{ij})
}{\big(1-h^2(g^{ij}\partial_iv_0 \partial_jv_0)\big)^2\mathrm{det}g} \right].
\end{split}
\end{equation*}
Since the two metrics above are equal, the formula (\ref{curv}) follows
directly. The formula for $\mbox{det}g$ is obtained by a direct
calculation, via $\mbox{det} (A+B) = \mbox{det} A +\mbox{cof} A : B +
\mbox{det} B$, valid for $2\times 2$ matrices $A, B$.

\medskip

In the general case when $v_0, v_1$ are only $\mathcal{C}^{2,\beta}$
regular, one may approximate them by smooth sequences $v_0^n, v_1^n$.
Then, each $\kappa_n = \kappa\big(\mbox{Id} + h^2(\nabla v_0^n\otimes
v_0^n - \nabla v_1^n\otimes v_1^n)\big)$ is given by the formula in
(\ref{curv}), and the sequence $\kappa_n$ converges in
$\mathcal{C}^{0,\beta}$ to the right hand side in (\ref{curv}). On
the other hand, $\kappa_n$ converges in $\mathcal{D}'(\Omega)$ to
$\kappa(g)$, which follows from the definition of Gauss curvature
$\kappa = {R_{1212}}/{\mbox{det} g}$. Hence the lemma is proven.
\end{proof}

\medskip

{\bf 2.} Applying  Lemma \ref{lem8.2} to $v_1 = v+ z_h$, we now see
that for small $h$, the Gauss curvature of metric $g_h(z_h)$ vanishes:
\begin{equation}\label{curv0}
\kappa(g_h(z_h)) = 0
\end{equation}
if and only if:
\begin{equation}\label{impl}
\Phi(h, z_h) = 0,
\end{equation}
where:
\begin{equation*}
\begin{split}
\Phi(h,z) = &\big(1-h^2|\nabla v + \nabla z|^2\big)^2
\mbox{det}\big(\nabla^2 v_0 - [\Gamma_{ij}^k\partial_k v_0]_{ij}\big)\\
&\qquad\qquad 
- \big(1-h^2(g^{ij}\partial_iv_0 \partial_jv_0)\big)^4 d(h,z) \mbox{det}
(\nabla^2v + \nabla^2z).
\end{split}
\end{equation*}
Here: 
$$d(h,z) = 1- h^4|(\nabla v_0)^\perp \cdot \nabla (v+z)|^2 + h^2
(|\nabla v_0|^2 - |\nabla v+\nabla z|^2)$$ 
and $\Gamma_{ij}^k$
and $g^{ij}$ are given by (\ref{chri}) and (\ref{inv}) for the metric  $g =
\mbox{Id} + h^2\nabla v_0 \otimes \nabla v_0 - h^2 (\nabla v + \nabla
z)\otimes (\nabla v + \nabla z)$.
We shall consider:
$$\Phi:(-\epsilon,\epsilon) \times
\mathcal{C}^{2,\beta}_0(\bar\Omega,\mathbb{R}) \longrightarrow 
\mathcal{C}^{0,\beta}(\bar\Omega,\mathbb{R}) $$
and seek for solutions $z_h \in
\mathcal{C}^{2,\beta}_0(\bar\Omega,\mathbb{R})$  of (\ref{impl}) with
zero boundary data.  It is
elementary to check that $\Phi$ is continuously Frechet differentiable
at $(0,0)$ and that 
$$\Phi(0,0) = \mbox{det} \nabla^2v_0 - \mbox{det} \nabla^2v = 0.$$
Moreover, the partial Frechet derivative $\mathcal{L} = \partial
\Phi/\partial z (0,0) : \mathcal{C}_0^{2,\beta}(\bar\Omega,\mathbb{R})
\longrightarrow \mathcal{C}^{0,\beta}(\bar\Omega,\mathbb{R})$ is a linear
continuous operator of the form:
\begin{equation*}
\begin{split}
\forall z\in\mathcal{C}_0^{2,\beta} \qquad
\mathcal{L}(z) &= \lim_{\epsilon\to 0} \frac{1}{\epsilon}
\Phi(0,\epsilon z) = \lim \frac{1}{\epsilon} \big(\det\nabla^2v_0 -
\det (\nabla^2v+\epsilon\nabla^2z)\big) \\
&= \lim \frac{1}{\epsilon}\big(
-\epsilon^2\det\nabla^2 z - \epsilon\mbox{cof} \nabla^2 v : \nabla^2
z\big) = -\mbox{cof} \nabla^2v:\nabla^2 z.
\end{split}
\end{equation*}
Clearly, $\mathcal{L}$ is invertible to a continuous linear operator,
because of the uniform ellipticity of the matrix field $\nabla^2v$
which follows from the convexity assumption of $\det\nabla^2v =
\det\nabla^2v_0$ being strictly positive. Thus, invoking the implicit
function theorem we obtain the solution operator:
$$\mathcal{Z}:(-\epsilon, \epsilon)\longrightarrow
\mathcal{C}_0^{2,\beta}(\bar\Omega,\mathbb{R})$$  
such that $z_h = \mathcal{Z}(h)$ satisfies (\ref{impl}).
Moreover, $\mathcal{Z}$ is differentiable at $h=0$ and:
$$\mathcal{Z}'(0) = \mathcal{L}^{-1}\circ \left(\frac{\partial
    \Phi}{\partial h} (0,0)\right) = 0,$$
because:
\begin{equation*}
\frac{\partial \Phi}{\partial h} (0,0) = \big(\mbox{cof}
\nabla^2v_0\big) : \left[(\frac{\partial}{\partial h}
  \Gamma_{ij}^k)\partial_kv_0\right]_{ij}
+ \frac{\partial}{\partial h}\det [\Gamma_{ij}^k\partial_kv_0]_{ij}
- \left(\frac{\partial}{\partial h} d(0,0)\right) \det \nabla^2 v =0.
\end{equation*}
Consequently:
\begin{equation}\label{wh3}
\|w_h^3\|_{\mathcal{C}^{2,\beta}} = \frac{1}{h}
\|z_h\|_{\mathcal{C}^{2,\beta}} \to 0 \quad \mbox{ as } h\to 0.
\end{equation}

\medskip

{\bf 3.} In conclusion, we have so far obtained a uniformly bounded sequence of
$\mathcal{C}_0^{2,\beta}$ out-of-plane displacements $w_h^3 = z_h/h$
such that the Gauss curvature (\ref{curv0}) of the metric $g_h(z_h)$
in the right hand side of (\ref{metric3}) is $0$.
By the result in \cite{mardare} it follows that for each small $h$
there exists exactly one (up to fixed rotations) orientation preserving isometric immersion
$\phi_h\in\mathcal{C}^2(\bar\Omega, \mathbb{R}^2)$ of $g_h(z_h)$:
\begin{equation}\label{isomh}
\nabla\phi_h^T\nabla\phi_h = g_h(z_h)  \quad \mbox{and} \quad \det\nabla\phi_h>0.
\end{equation}
What remains to be proven is that, in fact, $\phi_h = \mbox{id}
+h^2w_{h,tan}$ with some $w_{h,tan}$ uniformly bounded in
$\mathcal{C}^{2,\beta}(\bar\Omega,\mathbb{R}^2)$. 

It is a well known calculation (see \cite{ciar, mardare})
that (\ref{isomh}) implies (is actually equivalent to):
\begin{equation}\label{pdes}
\nabla^2\phi_h - [\Gamma_{ij}^k \partial_k\phi_h]_{ij} = 0,
\end{equation}
where $\Gamma_{ij}^k$ are the Christoffel symbols (\ref{chri}) of the
metric $g=g_h(z_h)$ in (\ref{8.4}). By (\ref{isomh})
$\|\nabla\phi_h\|_{L^\infty}\leq C$, and by (\ref{pdes})
$\|\nabla^2\phi_h\|_{L^\infty}\leq C$, hence
$\|\phi_h\|_{\mathcal{C}^{2,\beta}}\leq C$. But $\Gamma_{ij}^k$ are
uniformly bounded (with respect to small h) in
$\mathcal{C}^{0,\beta}$ so by (\ref{pdes})
$\|\nabla^2\phi_h\|_{\mathcal{C}^{0,\beta}}\leq C$ and thus:
$$\|\phi_h\|_{\mathcal{C}^{2,\beta}(\bar\Omega,\mathbb{R}^2)}\leq C.$$
Note now that $\|\Gamma_{ij}^k\|_{\mathcal{C}^{0,\beta}}\leq Ch^2$ in
view of the particular structure of the metrics $g_h(z_h)$. Hence, by
(\ref{pdes}):
\begin{equation}\label{sec}
\|\nabla^2\phi_h\|_{\mathcal{C}^{0,\beta}}\leq Ch^2.
\end{equation}
Therefore, for some $A_h\in\mathbb{R}^{2\times 2}$ we have:
\begin{equation}\label{cc}
\|\nabla\phi_h - A_h\|_{\mathcal{C}^{1,\beta}}\leq Ch^2.
\end{equation}
We now prove that the matrix $A_h$ in the inequality above can be
chosen as a rotation and hence, without loss of
generality,  $A_h = \mbox{Id}$.
For each $x\in\Omega$ there holds:
\begin{equation}\label{ccc}
\mbox{dist}(A_h, SO(3)) \leq |A_h - \nabla\phi_h(x)| +
\mbox{dist}(\nabla\phi_h(x), SO(3)).
\end{equation}
To evaluate the last term above, 
write: $\sqrt{\nabla\phi_h^T(x)\nabla\phi_h(x)} = QDQ^T$ for some
$Q\in SO(3)$ and $D=\mbox{diag}(\lambda_1,\lambda_2)$ with
$\lambda_1,\lambda_2>0$. Since $\det\nabla\phi_h>0$, it follows by
polar decomposition theorem that:
\begin{equation*}
\begin{split}
\mbox{dist}(\nabla\phi_h(x), SO(3)) & =
|\sqrt{\nabla\phi_h^T(x)\nabla\phi_h(x)} - \mbox{Id}| \leq
C|D-\mbox{Id}| \\ & = C\max_{i}\{|\lambda_i - 1|\} \leq C\max_i
\{|\lambda_i^2 - 1|\} \leq C|D^2 - \mbox{Id}| \\ & = C |Q^T
\nabla\phi_h^T(x)\nabla\phi_h(x) Q - \mbox{Id}| \leq
C|\nabla\phi_h^T\nabla\phi_h(x) - \mbox{Id}| \leq Ch^2.
\end{split}
\end{equation*}
By the above and (\ref{ccc}), (\ref{cc}) we see that $\mbox{dist}(A_h,
SO(3))\leq Ch^2$. Hence, without loss of generality, $\|\nabla\phi_h -
\mbox{Id}\|_{\mathcal{C}^{1,\beta}}\leq Ch^2$ and:
$$ \|\phi_h - \mbox{id}\|_{\mathcal{C}^{2,\beta}} \leq Ch^2.$$
Consequently, $\phi_h = \mbox{id} + h^2w_{h,tan}$ with $\|w_{h,
  tan}\|_{\mathcal{C}^{2,\beta}} \leq C$. This concludes the proof
of Theorem \ref{th8.1}, in view of (\ref{metric3}) which is equivalent to (\ref{metric}).
\endproof

\begin{remark}
The above proof is somewhat similar to \cite[Theorem 4.1.1]{hanhong}. 
In analogy, note the similarity between the proof of the matching
property in \cite{lemopa_convex}  and the Weyl problem by
Nirenberg in \cite{Ni}.  
\end{remark}

\section{Density and regularity for elliptic 2d Monge-Amp\`ere
  equation: a proof of Theorem \ref{thm-density}.} 

In this and the following sections, $\Omega \subset \R^2$ is a domain, i.e. an open, bounded and
connected set. 
The main step towards proving Theorem \ref{thm-density} will be the
following result in which we combine some key observations by
\v{S}ver\'ak from his unpublished manuscript 
\cite{Sve}, with a theorem due to Iwaniec and \v{S}ver\'ak in
\cite{IwanSve} (Theorem \ref{thm-IS})
regarding  deformations with integrable dilatation in dimension $2$.  

\begin{theorem}\label{thm-convex}
Let $u\in W^{2,2}(\Omega)$ be such that:
\begin{equation}\label{hesf}
\det \nabla^2 u = f ~~~~ \mbox{ in } \Omega, \quad \mbox{ where }
f:\Omega\rightarrow \mathbb{R}, \quad f(x) \ge c_0>0 \quad \forall a.e.~
x\in \Omega. 
\end{equation}
Then $u\in\mathcal{C}^1(\Omega)$ and, modulo a global sign change, $u$ is locally convex in $\Omega$. 
\end{theorem}

\begin{example}\label{example} 
Let $B_1$ be the unit disk in $\R^2$ and let $u\in\mathcal{C}^1(B_1)$ be given by: 
\begin{equation*} 
u(x,y) = \left \{ \begin{array}{ll}
x^2 e^{y^2/2} & \mbox{if} \,\, x\ge 0 \\
-x^2 e^{y^2/2} & \mbox{otherwise.}
\end{array} \right .
\end{equation*} 
Note that $u(0,y) = 0$ and  $\nabla u (0,y) =0$ for all $y\in (-1,1)$. 
Indeed, we have $u_x = \pm 2x e^{y^2/2}$, $u_y = \pm y x^2 e^{y^2/2}$, 
$u_{xx}= \pm 2 e^{y^2/2}$,  $u_{xy} = u_{yx} = \pm 2xy e^{y^2/2}$, $u_{yy} = \pm (x^2 e^{y^2/2} + y^2 x^2
e^{y^2/2})$ and  $\Delta u= \pm e^{y^2/2} ( 2 + x^2  + y^2 x^2)$, respectively for $x>0$ and $x<0$. 

As a consequence $u\in W^{2,\infty}(B_1)$, $u$ is strictly convex in $\{(x,y)\in
B_1; \, x>0\}$ and 
strictly concave in $\{(x,y)\in B_1; \, x<0\}$. On the other hand 
$\det \nabla ^2 u=  2x^2 e^{y^2} ( 1- y^2) \in \mathcal{C}^\infty(B_1)$ and it is positive 
if $x\neq 0$ and  $y^2 < 1$. We right away note that
$u\not\in\mathcal{C}^2(B_1)$ although it solves the Monge-Amp\`ere equation with
smooth non-negative right hand side, a.e. in its domain.

Finally, our example shows that the assumption of strict
positivity in Theorem \ref{thm-convex} (and also \ref{thm-IS} and
Theorem \ref{thm-regularity} below) cannot
be relaxed to $\det \nabla^2 u>0$ a.e., even assuming a better $W^{2,\infty}$
regularity for $u$.   
\end{example} 

As a consequence of Theorem \ref{thm-convex} and of the monotonicity
property by Vodopyanov and Goldstein which we quote in Theorem
\ref{modulus}, we obtain:
 
\begin{theorem}\label{thm-regularity}
Let $f\in \mathcal{C}^{k,\beta}(\Omega)$ be a positive function. 
Then any $W^{2,2}(\Omega)$ solution of $\det \nabla^2 u=f$, is
$\mathcal{C}^{k+2, \beta}$ regular, locally in $\Omega$.   
\end{theorem} 
\begin{proof}   
{\bf 1.} We first note that $u$ is a generalized Aleksandrov solution
to (\ref{hesf}) . Since $u$ is locally convex, it is twice
differentiable in the classical sense a.e. in $\Omega$, and its gradient agrees
with $f$. By Lemma 2.3 in \cite{TW}, the regular part of the
Monge-Amp\`ere measure $\mu_u$ equals $(\det\nabla^2 u) dx = f dx.$
It suffices now to prove that there is no singular part of $\mu_u$,
i.e. that $\mu_u$ is absolutely continuous with respect to the
Lebesgue measure $dx$.

Call $v=\nabla u$. By Theorem \ref{thm-convex} we have
$u\in\mathcal{C}^1(\Omega)$ and hence:
$$\mu_u(\omega) = |v(\omega)| \qquad \mbox{ for every Borel set } 
\omega\subset\Omega.$$ 
We thus need to show that $v$ satisfies Luzin's
condition (N): 
$$|v(\omega)|=0  \qquad \forall \omega\subset\Omega; \quad 
|\omega|=0.$$ 
The above claim follows directly from Theorem A in
\cite{Maly}, in view of $v\in W^{1,2}(\Omega,\mathbb{R}^2)$ and the
monotonicity property (\ref{mon}) of $v$ due to Vodopyanov and Goldstein.

\medskip

{\bf 2.} Since $f\in\mathcal{C}^{0,\alpha}(\bar\Omega)$, then Theorem
5.4 in \cite{huang} implies that $u$ is locally
$\mathcal{C}^{2,\alpha}$.  We note that this statement is the well-known result due to
Caffarelli \cite{Caffarelli}.  Indeed, fix $x_0\in\Omega$. By Remark 3.2
in \cite{TW} which gives an elementary proof of a result by
Aleksandrov and Heinz, the displacement $u$ as above must be
strictly convex in some $B(x_0,\epsilon)$.
By adding an affine function to $u$, we may without loss of generality
assume that $u=0$ on the boundary of the convex set: 
$$\Omega_0 = \{x\in\Omega; ~ u(x) \leq u(x_0)+\delta\}\subset
B(x_0,\epsilon),$$ for a sufficiently small $\delta>0$.
Therefore, the statement of Theorem 5.4 in \cite{huang} can be
directly applied. 
Once the $\mathcal{C}^{2,\beta}$ regularity is established, the
$\mathcal{C}^{k+2,\beta}$ regularity follows as in Proposition 9.1 in \cite{Cafcab}.
\end{proof}
 
\bigskip

\noindent {\bf Proof of Theorem \ref{thm-density}.} 
Without loss of generality we assume that $\Omega$ is starshaped with
respect to $B=B(0,r)\subset \Omega$.  Let $u\in {\mathcal A}$ and define
$u_\lambda(x)= \frac 1{\lambda^2}u(\lambda x)$ for $0<\lambda<1$. 
Then:
\begin{equation*}
\det \nabla^2 u_\lambda(x)= c_0 \quad\forall a.e.~ x\in\Omega
\end{equation*}
and:
$$ \| u_\lambda\|_{L^2(\Omega)} =\lambda^{-3} \| u\|_{L^2(\lambda
  \Omega)}, \quad
\| \nabla u_\lambda\|_{L^2(\Omega)} =\lambda^{-2} \| u\|_{L^2(\lambda
  \Omega)}, \quad 
\| \nabla^2 u_\lambda\|_{L^2(\Omega)} =\lambda^{-1} \| u\|_{L^2(\lambda
  \Omega)}.$$
As a consequence, $u_\lambda \in {\mathcal A}$ for all 
$\lambda\in (0,1)$, and  $u_\lambda \to u$ strongly in $W^{2,2}(\Omega)$ as $\lambda\to 1^-$.   

So far we have used only the fact that $\Omega$ is starshaped with
respect to the origin $0$. Now, since $\Omega$ is star-shaped with
respect to an open ball $B$, we have $\lambda \bar\Omega \subset
\Omega$ for all $0< \lambda<1$.
%
%
Hence, in view of Theorem \ref{thm-regularity},  $u_\lambda\in
\mathcal{C}^\infty(\bar \Omega) \cap {\mathcal A}$, which proves the claim. 
\endproof

\section{\v{S}ver\'ak's arguments: a proof of Theorem \ref{thm-convex}.}

\begin{definition}\label{defi1}
We say that a mapping $v\in \mathcal{C}^0(\Omega, \R^2)$ is connectedly locally one-to-one
iff it is locally one-to-one outside of a closed set $S\subset \Omega$
of measure zero, for which $\Omega\setminus S$  is connected. 
\end{definition} 

\begin{definition}\label{intdil}
Let $v\in W^{1,2} (\Omega, \R^2)$ and let $\det\nabla v \ge  0$
a.e. in $\Omega$. 
We  say that $v$ has integrable dilatation iff: 
$$ \forall a.e.~ x\in \Omega \qquad |\nabla v|^2 (x) \le K(x) \det \nabla v(x)  $$ 
with some function $K \in L^1(\Omega)$.
\end{definition} 
 
For the proof of Theorem \ref{thm-convex}, the first result we propose is essentially a combination of arguments
in \v{S}ver\'ak's unpublished paper \cite{Sve}.  In this section, we will gather all the
details of its proof.

\begin{theorem}[\v{S}ver\'ak]\label{thm-convex2}
If $u\in W^{2,2}(\Omega)$ satisfies: 
\begin{equation}\label{hespos}
\det \nabla^2 u (x) >0 \quad \forall a.e.~ x\in\Omega,
\end{equation}
then $u\in \mathcal{C}^1(\Omega)$. If additionally  $v= \nabla u$ is connectedly locally
one-to-one, then modulo a global sign change, $u$ is locally
convex  in $\Omega$. In particular, when $\Omega$ is convex then $u$
is either convex or concave in the whole $\Omega$. 
\end{theorem}

We quote now  the result, which will be crucial for the proof of Theorem \ref{thm-convex2}:

\begin{theorem}[Iwaniec and \v{S}ver\'ak \cite{IwanSve}]\label{thm-IS}
Let $v\in W^{1,2}(\Omega, \mathbb{R}^2)$ be as in Definition \ref{intdil}.
Then there exists a homeomorphism $h\in W^{1,2}(\Omega', \Omega)$
and a holomorphic function $\varphi \in W^{1,2}( \Omega' ,\R^2 
={\mathbb C})$ such that: 
$$ v = \varphi \circ h^{-1}. $$ 
In particular, $v$ is either constant or connectedly locally
one-to-one, and in the latter case the singular set $S= h((\nabla\phi)^{-1}\{0\})$.
\end{theorem}

Without having Theorem \ref{thm-IS} at hand, \v{S}ver\'ak proved in \cite{Sve} that
if $u\in W^{2,2}(\Omega)$ satisfies $\det \nabla^2 u>0$ a.e. in $\Omega$, then there exists a closed set
$S\subset \Omega$ of measure 
zero such that on each component of $\Omega \setminus S$, $u$ is
either locally convex or locally concave. In fact, the main step in
the proof is to show that any such map
is locally one-to-one outside a set of measure zero, which \v{S}ver\'ak
has achieved by using consequences of a version of  Lemma
\ref{controlofu} below and the classical degree theory.

\medskip

Combining Theorem \ref{thm-convex2} with Theorem \ref{thm-IS} one
directly obtains: 
 
\begin{corollary}\label{thm-convex3}
Let $u\in W^{2,2}(\Omega)$ satisfy (\ref{hespos}) and  be such that $\nabla u$ has integrable dilatation. 
Then $u\in \mathcal{C}^1(\Omega)$ and modulo a global sign change, $u$ is locally convex 
in $\Omega$. 
\end{corollary}

Theorem \ref{thm-convex} is then, obviously, a particular case of the above corollary, where the displacement $u$
trivially satisfies its assumptions.
In the remaining part of this section, we will prove Theorem \ref{thm-convex2}.
We first remind a key result on the modulus of continuity of 2d deformations in
$W^{1,2}$ with positive Jacobian:

\begin{theorem}[Vodopyanov and Goldstein \cite{VodoGold}]\label{modulus}
Assume that $v\in W^{1,2}(\Omega, \R^2)$ and that $\det \nabla v >0$ a.e. in
$\Omega$. Then $v$ is continuous in $\Omega$, and for any $B(x,\delta) \subset
  B(x,R) \subset \Omega$ we have:
\begin{equation}\label{modest}
{\rm osc}_{B(x,\delta)} v  \le  \sqrt{2\pi} (\ln \frac R\delta)^{-1/2}
\|\nabla v\|_{L^2(B(x,R))}.
\end{equation} 
\end{theorem} 
\begin{proof}
By a result of Vodopyanov and Goldstein \cite{VodoGold}
$v$ is continuous (see also \cite{Man, FonsGang, Sve2}). In fact, a
key ingredient of this result is to show that $\phi$ is a monotone map,
i.e. for $B_\rho= B(x,\rho)$: 
\begin{equation}\label{mon} 
{\rm osc}_{B_\rho} v = {\rm osc}_{\partial B_\rho} v, 
\end{equation}
and hence $v$ has the asserted modulus of continuity by
\cite [Theorem 4.3.4]{Mor} (see also \cite{Man, FonsGang}). We sketch the last
part of the proof for the convenience of the reader. By Fubini's
theorem $v$ belongs to $W^{1,2}(\partial B_\rho)$ for almost every
$\rho\in (\delta, R)$. Hence the Morrey's theorem of embedding of $W^{1,2}$ into
$\mathcal{C}^0$ for the one-dimensional 
set $\partial B_\rho$ yields:
$$ \forall a.e.~\rho \in (\delta, R) \qquad {\rm osc}_{B_\delta} v \le {\rm osc}_{B_\rho} v = {\rm
  osc}_{\partial B_\rho} v \le \sqrt{2\pi\rho} \Big (
\int_{\partial {B_\rho}} |\nabla v |^2 \Big )^{1/2} . $$ 
To conclude, one squares both sides of the above inequality, divides by $\rho$ and integrates from
$\delta$ to $R$, in order to deduce (\ref{modest}). 
\end{proof}

\begin{corollary}\label{convergence1} 
Assume that $v_n\in W^{1,2}(\Omega,\mathbb{R}^2)$ is a bounded
sequence such that $\det \nabla v_n >0$ a.e. in
$\Omega$. Then, up to a subsequence, $v_n$ converges locally uniformly and
also weakly in $W^{1,2}$ to a continuous mapping $v \in
W^{1,2}(\Omega, \R^2)$ satisfying $\det \nabla v \ge 0$ a.e. in $\Omega$. 
\end{corollary}
\begin{proof}
The uniform convergence of a subsequence follows by Ascoli-Arzel\'a
theorem in view of Theorem \ref{modulus}. Noting that 
$\det\nabla v = - \nabla v_1 \cdot \nabla^\perp v_2$, the Div-Curl Lemma implies then
that the desired inequality is satisfied for the limit mapping $v$.
%
\end{proof}

\begin{corollary}\label{convergence2} 
Let $u_n\in W^{2,2}(\Omega)$ be a bounded sequence 
such that $\det \nabla^2 u_n \ge c_0 \in \R$ a.e. in
$\Omega$. Then, up to a subsequence, $u_n$ converges  weakly in
$W^{2,2}$, as well as it converges locally uniformly together with
its gradients, to a $\mathcal{C}^1$ function $u \in
W^{2,2}(\Omega, \R^2)$ satisfying $\det \nabla^2 u \ge c_0$  a.e. in $\Omega$.
\end{corollary}
\begin{proof}
Let $v_n(x)=\nabla u(x) + (|c_0|+1)^{1/2}x^\perp$, where
$x^\perp = (x_1, x_2)^\perp = (-x_1, x_2)$. Clearly, $v_n\in
W^{1,2}(\Omega,\mathbb{R}^2)$  and, since $\nabla ^2 u_n$ is a
symmetric matrix, we get:
$$ \det\nabla v_n(x) = \det \nabla^2 u_n(x) + |c_0| + 1>0 \qquad \forall a.e.~ x\in\Omega.$$
The convergence assertion follows by Corollary \ref{convergence1}.
Again, the Div-Curl Lemma applied to sequence $\nabla u_n$  implies the desired inequality
for the limit function $u$. 
\end{proof}

A consequence of Theorem \ref{modulus} is the following
assertion about $W^{2,2}$ functions whose Hessian determinants are
uniformly  controlled from below:  

\begin{lemma}\label{controlofu} 
Assume that $u\in W^{2,2}(\Omega)$ satisfies:
\begin{equation}\label{hesc0}
\det \nabla^2 u(x)\geq c_0 \in\mathbb{R} \quad \forall a.e.~ x\in \Omega. 
\end{equation}
Then $u\in \mathcal{C}^1(\Omega)$. Moreover, if $x_0\in \Omega$ is  a Lebesgue
point for $\nabla^ 2 u$,  i.e. for some $A\in \R^{2\times 2}_\sym$: 
\begin{equation}\label{leb}
\omega(r):= \frac{1}{|B(x_0,r)|} \int_{B (x_0,r)} |\nabla^2 u - A|^2
{\rm d}x \to 0 \quad \mbox{as} \quad r\to 0^+,     
\end{equation}
then for all $\epsilon>0$ there exists $r_0>0$   such that:
\begin{equation}\label{graduest}
\begin{split}
\forall r<r_0\quad &\forall a\in D_r= \overline{B(x_0, r)} 
\qquad  \|\nabla u (x)- \nabla u(a) - A(x-a)\|_{\mathcal{C}^0(D_{r})}
\le \frac 12 \epsilon r,  \\
&\|u(x) - u(a) - \nabla u(a) \cdot (x-a) - \frac 12 (x-a)\cdot
A(x-a) \|_{\mathcal{C}^0(D_r)} \leq \epsilon r^2.
\end{split}
\end{equation}  
\end{lemma}
\begin{proof}
Following Kirchheim \cite{Kir} we set $v=\nabla u$
and we write $\phi(x)= v(x) + (|c_0|+1)^{1/2}x^\perp$. Trivially $\phi \in
W^{1,2}(\Omega, \R^2)$ and, as before:
$$ \det \nabla \phi = \det \nabla^2 u + |c_0|+1 >0 \qquad \forall a.e.~
x\in\Omega. $$ 
Applying Theorem \ref{modulus} to $\phi$ shows that $v$ is
continuous and so  $u\in \mathcal{C}^1(\Omega)$.  

In what follows, we assume without loss of generality that $x_0=0$,
$u(0)=0$ and $v(0)=\nabla u(0)= 0$ (otherwise it is sufficient to translate
$\Omega$ and to  modify $u$ by its tangent map at $0$). 
For $r$ sufficiently small and for all $x\in B_2=B(0,2)$ we define:
$$ v_r (x):= \frac 1r v (r x),  \qquad \phi_r (x) := v_r (x) + (|c_0|+1)^{1/2} x^\perp,$$
so that:
\begin{equation*} 
\begin{split}
\forall x\in B_2\qquad &\nabla v_r  (x) = \nabla v(r x) = \nabla^2 u(r x),
\qquad \nabla \phi_r (x)
= \nabla^2 u(r x) + (|c_0|+1)^{1/2} \left [ \begin{array}{cc} 0 & -1 \\ 1
    & 0 \end{array} \right ],  \\
&\det  \nabla \phi_r (x)= \det \nabla v_r (x)+ (|c_0| + 1)>0.
\end{split}
\end{equation*}
Since $\phi_r \in W^{1,2}(B_2,\R^2)$, we can apply
(\ref{modest}) to $x\in B_1$ and  $\delta<R=1$, to obtain for $r<r_0$ small
enough:  
\begin{equation*} 
\begin{split}
{\rm osc}_{B(x,\delta)} \phi_r & \le \sqrt{2\pi} (\ln \frac 1\delta )^{-1/2}
\Big ( \int_{B(x,1)} |\nabla \phi_r|^2\Big )^{1/2}  \\ & 
\leq \sqrt{2\pi} (\ln \frac 1\delta)^{-1/2} 
\Big(2 (|m|+1)^{1/2} |B_1|+
(\int_{B(x,1)} |\nabla^2 u(ry)|^2~\mbox{d}y)^{1/2} \Big) \\
& \leq  C \ln (\frac 1\delta)^{-1/2}
\Big((|m|+1)^{1/2} + \frac{1}{r}\|\nabla^2 u \|_{B(0,2r)}\Big)\\
& \leq C \ln (\frac 1\delta)^{-1/2},
\end{split}
\end{equation*} 
where $C=C(m,|A|)>0$. Above we used the fact that $B(x,1)\subset B_2$
and that $0$ is a Lebesgue point for $\nabla^2 u$. 
Now, given $\epsilon >0$ we choose $\delta >0$ such that: 
$$ \ln (\frac 1\delta)^{-1/2} < \epsilon /C. $$ 
Consequently: 
$$\forall x,y\in D_1 \quad \forall r<r_0 \qquad |x-y|<\delta \implies
|\phi_r(x)-\phi_r(y)|<\epsilon.$$ 
Since $v_r  - \phi_r$ is a given linear deformation, we conclude that  the family:
$$ {\mathcal F}=\{v_r : D_1 \to \R^2; ~~ r<r_0\} $$ 
is equicontinuous. On the other hand:
$$ \int_{D_1} \Big |\nabla v_r -   A \Big |^2  = \pi \omega(r) \to 0
\quad \mbox{ as } r\to 0.  $$ 
Let $\tilde v_r= v_r - \dashint_{B_1} v_r$ and apply
the Poincar\'e inequality to obtain that 
$$ \tilde v_r \to Ax \quad \mbox{ in } W^{1,2}(\Omega, \R^2) \quad \mbox{ as } r\to 0.$$  
Now, equicontinuity of  ${\mathcal F}$ and $v_r(0)=0$ yield, by
Arzel\`a-Ascoli theorem, that a subsequence of $v_r$ (which we do not
relabel) converges uniformly to a continuous function $V$ on $D_1$. 
Since $v_r - \tilde v$ is constant, we deduce that $V(x) -Ax=c$ 
is constant too. But then, evaluating at $0$ gives $c=0$. Hence, $v_r$ uniformly converges
to $Ax$ on $D_1$. 

Let us fix $\epsilon>0$ and choose $r_0$ so that:
$$ \forall r< r_0 \qquad \|v_r (x) - Ax\|_{\mathcal{C}^0(D_1)} \leq \frac 14 \epsilon.  $$ 
This implies: 
$$\|\nabla u(x) - Ax\|_{\mathcal{C}^0(D_r)} \le \frac 14 \epsilon r. $$ 
Fixing $a\in D_r$ we get:
$$ \|\nabla u(x) - \nabla u(a) - A (x-a)\|_{\mathcal{C}^0(D_r)} \le
\frac 12 \epsilon r, $$ 
giving the first estimate in (\ref{graduest}). Since
${\mathrm{diam}} ~D_r=2r$, the second estimate follows.
\end{proof}


\medskip

We now prove a simple useful lemma whose statement we quote from \cite{Sve}:

\begin{lemma} [\cite{Sve}, Lemma 2]\label{convexiseasy} 
Let  $u\in \mathcal{C}^1(\bar\Omega)$ and fix $a\in \Omega$. Suppose
that:
\begin{equation}\label{sth}
u(x) \ge u(a) + \nabla u(a) \cdot (x-a) \quad \forall x\in \partial
\Omega
\end{equation}
and that $\nabla u(x) \neq \nabla u(a)$ for all $x\in
\Omega\setminus\{a\}$. Then $u$ has a supporting hyperplane at
$a$, i.e. (\ref{sth}) holds for all $x\in \Omega$. In particular, 
if $\Omega$ is convex then:
$$({\mathrm C} u)(a) = u(a)$$ 
where ${\mathrm C} u$
denotes the convexification of the function $u$   over $\bar\Omega$:
$$({\mathrm C}u)(a) = \sup \left\{T(a); ~~T:\Omega\to \R 
\mbox{ is affine and }  T(x) \le u(x) ~~ \forall x\in \bar\Omega \right\}.  $$  
\end{lemma}
\begin{proof}
Consider the tangent map $T(x)= u(a) + \nabla u(a) \cdot (x-a)$. We now
claim that $T(x) \le u(x)$ for all $x\in \bar\Omega$. For otherwise, 
the continuous function $g(x) = u(x) -T(x)$ would 
assume a negative minimum on $\bar\Omega$ at some
$c\in \Omega\setminus\{a\}$. Hence $\nabla g(c)=0$,
which is a contradiction with the second assumption as $\nabla u(c)= \nabla T(c)=
\nabla u(a)$. 
\end{proof}

\bigskip

We are ready to prove the key theorem of this section:

\bigskip

\noindent {\bf Proof of Theorem \ref{thm-convex2}.}

{\bf 1.} The $\mathcal{C}^1$ regularity of $u$ is an immediate consequence
of Theorem \ref{modulus}. Recall the properties of the singular set $S$
from Definition \ref{defi1}.
Since $\det \nabla^2 u>0$
a.e. in $\Omega$, modulo a global change of
sign for $u$ we can choose $x_0\in \Omega\setminus S$ a  Lebesgue
point of $\nabla^2 u$ as in (\ref{leb}), such that the matrix $A$ is
positive definite.  Hence there exists $\lambda >0$ for which $\xi
\cdot A \xi \ge \lambda |\xi|^2$ for all $\xi\in \R^2$. 
By Lemma \ref{controlofu} for all $r<r_0$ the estimate  (\ref{graduest})
holds true with $\epsilon =\frac 14  \lambda$, and without loss of
generality $\nabla u$ is also one-to-one on
$D_r=\overline{B(x_0, r)}\subset \Omega\setminus S$. By  (\ref{graduest}) it follows that:
$$\forall x\in\partial D_r\qquad  u(x) - u(x_0) - \nabla u(x_0) \cdot (x-x_0) \ge \frac 12 (x-x_0)
\cdot A(x-x_0) - \varepsilon r^2 \ge  \frac \lambda 4 r^2 >0.$$  
In view of  Lemma \ref{convexiseasy}, $u$ therefore admits a supporting
hyperplane at $x_0$ on $D_r$. 

\medskip

{\bf 2.} Our next  claim is that $u$ is locally convex
in $\Omega \setminus S$. Since $\Omega\setminus S$ is open and
connected, it  is also path-wise
connected. Therefore, for a fixed $x\in \Omega\setminus S$,  there exists a continuous path within
$\Omega\setminus S$ connecting $x$ and $x_0$, which  can be covered with a finite chain of
open balls  $B_i\subset \Omega \setminus S$, $i=1,2\ldots n$,
such  that  $x_0\in B_1$, $B_i \cap B_{i+1} \neq \emptyset$ and $x\in
B_n$. We now  need  the following  strong theorem due to J. Ball:

\smallskip

\noindent {\bf Theorem} ({\it \cite{Ball}, Theorem 1).
Let $\Omega\subset \R^n$ be open and convex, and let $u\in
\mathcal{C}^1(\Omega)$. The necessary and sufficient condition for $u$ to be strictly
convex on $\Omega$ is: 

(i) $\nabla u$ is locally one-to-one, and 

(ii) there exists a locally supporting hyperplane for $u$ at some point
of $\Omega$:
$$ \exists x_0\in \Omega \quad \forall \rho>0 \quad \forall x\in B(x_0,\rho)\qquad
u(x) \ge u(x_0) + \nabla u(x_0) \cdot (x-x_0).$$
}
\noindent Applying this result consecutively to each ball $B_i$,
we deduce that $u$ is strictly convex on $B_n$, hence it is
locally convex at $x$.

\medskip

{\bf 3.}  To finish the proof, we fix a direction in $\R^2$ and
consider the family  of straight lines parallel to that direction. For
almost every such line $L$, the 1-dimensional Lebesgue measure of $L\cap
S$ is zero and $u \in W^{2,2}  (L \cap \Omega)$. Also, $u$ is locally
convex on $(L\cap \Omega) \setminus S$ in view of the previously
proven claim. We now state the following easy 
lemma to show that $u$ is convex on connected components of $L\cap \Omega$:

\smallskip

\noindent{\bf Lemma.} {\it
Let $I\subset \R$ be an open bounded interval, $\phi\in W^{2,1} (I)$ and assume
that $\phi$ is locally convex on $I\setminus S$, where $S$ is a set of measure
$0$. Then $\phi$ is convex on $I$. }

\smallskip

\noindent
The proof is elementary. Since $\phi\in \mathcal{C}^1(I)$ and $\phi$ is locally convex on a full measure
open subset of $I$, we deduce that $\phi'' \ge 0$ a.e. in $I$. But this immediately implies
that $\phi'$ is increasing in $I$, hence $\phi$ is globally convex.

\medskip

We have previously shown that $u$ is convex on connected components of $L \cap
\Omega$,  for almost all straight lines $L$ in any direction. By
continuity of $\nabla u$,  the same must hold, in fact, for  all
lines, by approaching any given line with a selected
sequence of 'good' lines and passing to the limit in the convexity
inequality. This implies that actually $u$ is convex on any convex
subset of $\Omega$ and the proof is done.
\endproof

For completeness, we now note another corollary of Lemma  \ref{controlofu} and Lemma
\ref{convexiseasy}:

\begin{lemma}\label{extrabonus} 
Let $u\in W^{2,2}(\Omega)$ satisfy (\ref{hesc0}). Assume that $x_0$ is a Lebesgue point for $u$
with $A$ in (\ref{leb}) being positive definite. Assume that $\nabla u$ is one-to-one in
a neighborhood of $x_0$. Then $u$ is locally convex  at $x_0$.  
\end{lemma}  
\begin{proof}
There exists $\lambda >0$ for which $\xi \cdot A \xi \ge \lambda |\xi|^2$ for all $\xi\in \R^2$. 
By Lemma \ref{controlofu} for all $r<r_0$ and all $a\in
D_r=\overline{B(x_0, r)}$, estimate (\ref{graduest})
holds true with $\epsilon=\lambda/4$. Without loss of generality, $\nabla u$ is one-to one on
$D_r$ and,  by (\ref{graduest}):
\begin{equation*}
\begin{split}
\forall a\in B(x_0, r/2) & \quad \forall x\in \partial B(x_0,
r/2)\qquad\\ 
& u(x) - u(a) - \nabla u(a) \cdot (x-a) \ge \frac 12 (x-a)
\cdot A(x-a) - \varepsilon r^2 \ge  \frac \lambda{4} r^2 >0. 
\end{split}
\end{equation*} 
The assumptions of Lemma \ref{convexiseasy}
are satisfied and hence $u(a)=({\mathrm C}u)(a)$ for all $a\in B(x_0, r/2)$. The claim is proved. 
\end{proof}

\begin{remark} In proving Theorem \ref{thm-convex} we only used  the
conclusion of Theorem \ref{thm-IS} that
$v$ is locally one-to-one on a connected set of full measure. Therefore,
the assumptions of Theorem \ref{thm-convex} could potentially be
relaxed (as in Theorem \ref{thm-convex2}), but not to
$\det \nabla^2 u >0$ a.e. Indeed, let $v= \nabla u$ be as in Example \ref{example}. 
Then $v$ is not of integrable dilatation because:
$$\frac{|\nabla v|^2}{\det\nabla v} (x,y) \geq \frac{2}{x^2 (1-y^2)},$$
and also the singular set $S=\{(0,y); ~~ y\in (-1,1)\}$ coincides with the vanishing set of
$v$ where $v$ is obviously not locally one-to-one.  On the other hand,
Theorem \ref{thm-convex2} can be also
applied to the cases where $\det \nabla^2 u\in \mathcal{C}^0(\Omega)$
is positive a.e. and $\Omega\setminus f^{-1}(0)$ is connected.
\end{remark} 

\section{Recovery sequence: a proof of Theorem \ref{upbd}}\label{last}

It is enough to prove Theorem \ref{upbd} for
$v\in\mathcal{C}^{2,\beta}(\bar\Omega)$ satisfying $\det\nabla^2v =
\det\nabla^2 v_0$. In the general case of $v\in W^{2,2}(\Omega)$
satisfying the same constraint, the result follows then by a
diagonal argument in view of the density property established in
Theorem \ref{thm-density}.

\medskip

{\bf 1.} We now recall the useful change of variable
$\tilde\phi_h\in\mathcal{C}^{1,\beta}(\Omega^h, \mathbb{R}^3)$ between
thin plates $\Omega^h = \Omega\times (-{h}/2, h/2)$ and thin shallow
shells $(S_h)^h$:
$$ \tilde\phi_h(x, x_3) = x+h^\alpha v_0(x) e_3+x_3 \vec n^h(x)
\qquad \forall (x, x_3) \in\Omega,$$
where $\vec n^h$ is the unit normal vector to the midsurface $S_h$, given as
the image of the map $\phi_h(x) = x+ h^\alpha v_0(x)e_3$:
$$\vec n^h(x) = \frac{\partial_1\phi_h(x) \times \partial_2\phi_h(x)}{|\partial_1\phi_h(x)
  \times \partial_2\phi_h(x)|} = \frac{1}{\sqrt{1+h^{2\alpha}|\nabla
    v_0|^2}}\left(-h^\alpha (\nabla v_0)^* + e_3\right).$$
By Theorem \ref{th8.1}, there exists an equibounded sequence
$w_h\in\mathcal{C}^{2,\beta}(\bar\Omega,\mathbb{R}^3)$ such that the
deformations $\xi_h(x) = x+ h^\alpha v(x) e_3 + h^{2\alpha} w_h(x)$
are isometrically equivalent to ${id}+h^\alpha v_0e_3$:
\begin{equation}\label{isom}
\forall 0< h \ll 1 \qquad (\nabla \xi_h)^T\nabla\xi_h = \nabla(\mbox{id}+h^\alpha v_0
e_3)^T\nabla(\mbox{id} + h^\alpha v_0e_3).
\end{equation}
Define now the recovery sequence $u^h\in
\mathcal{C}^{1,\beta}((S_h)^h,\mathbb{R}^3)$ by:
$$v^h(x, x_3) = u^h(\tilde\phi_h(x, x_3)) = \xi_h(x) + x_3\vec N^h(x) +
\frac{x_3^2}{2}h^\alpha d^h(x),$$
where $\vec N^h$ is the unit normal vector to the image surface
$\xi_h(\Omega)$:
$$\vec N^h(x) = \frac{\partial_1\xi_h(x) \times \partial_2\xi_h(x)}{|\partial_1\xi_h(x)
  \times \partial_2\xi_h(x)|} = \left(-h^\alpha (\nabla v)^* +
  e_3\right) + \mathcal{O}(h^{2\alpha}),$$
while the 'warping' vector fields
$d^h\in\mathcal{C}^{1,\beta}(\bar\Omega,\mathbb{R}^3)$, approximating
the effective warping $d\in\mathcal{C}^{0,\alpha}(\bar\Omega,\mathbb{R}^3)$ 
are defined so that:
\begin{equation}\label{warp}
\begin{split}
 h^\alpha \|d^h\|_{\mathcal{C}^{1,\beta}} &\leq C \quad \mbox{ and } \quad  
\lim_{h\to 0} \|d^h - d\|_{L^\infty} = 0,\\
\mathcal{Q}_2\big(\nabla^2 v_0 - \nabla^2 v\big) & = \min\left\{\mathcal{Q}_3(F); ~~
F\in\mathbb{R}^{3\times 3}, ~ F_{tan} = \nabla^2 v_0 - \nabla^2
v\right\}  \\ & = \mathcal{Q}_3\big((\nabla^2 v_0 - \nabla^2 v)^* +
\mbox{sym}(d\otimes e_3)\big).
\end{split}
\end{equation}
For $F\in\mathbb{R}^{3\times 3}$, by $F_{tan}$ we denote the principal
$2\times 2$ minor of $F$. Recall also that the quadratic form
$\mathcal{Q}_3$ is given by $\mathcal{Q}_3(F) = D^2W(\mbox{Id})(F, F)$.

\medskip

{\bf 2.} Because of the first condition in (\ref{warp}), the
statements in Theorem \ref{upbd} (i), (ii) easily follow. In order to compute the energy
limit in (iii), we write:
\begin{equation}\label{nowe}
I^h(u^h)= \frac{1}{h} \int_{\Omega^h} W\left((\nabla v^h)(b^h)^{-1}\right)
\det\nabla\tilde\phi_h = \frac{1}{h} \int_{\Omega^h} W\left(\sqrt{K^h}\right)
\det b^h,
\end{equation}
where:
$$ b^h = \nabla \tilde\phi^h $$
while the frame invariance of $W$ justifies the second equality in
(\ref{nowe}) with:
$$K^h(x, x_3) = (b^h)^{-1, T} (\nabla v^h)^T (\nabla v^h) (b^h)^{-1}.$$
We will now compute the entries of the symmetric matrix field $K^h$, up to terms of
order $o(h^{\alpha+1})$. In what follows we adopt the convention that
all equalities hold modulo quantities which are uniformly $o(h^{\alpha+1})$.
Call $M^h = (\nabla v^h)^T \nabla v^h$. Since:
$$\nabla_{tan}v^h = \nabla\xi_h + x_3\nabla \vec N^h +
o(h^{\alpha+1}), \qquad \partial_3 v^h = \vec N^h + x_3 h^\alpha d^h,$$
we obtain, in view of (\ref{isom}): 
\begin{equation*}
\begin{split}
M^h_{tan} & = 
\nabla(\mbox{id}+ h^\alpha v_0e_3)^T \nabla(\mbox{id} + h^\alpha v_0e_3) +
2x_3 \mbox{sym}\left((\nabla\xi_h)^T\nabla \vec N^h\right) \\ & = \mbox{Id}_2
+ h^{2\alpha}\nabla v_0\otimes\nabla v_0 - 2x_3h^\alpha \nabla^2v +
o(h^{\alpha+1}),\\
M^h_{13, 23} & = (M^h)^T_{13, 23} = x_3 h^\alpha d^h_{tan}  + o(h^{\alpha+1}), \\  
M^h_{33} & = |\vec N^h + x_3h^\alpha
d^h|^2 = 1+ 2x_3h^\alpha d^h_3 + o(h^{\alpha+1}).
\end{split}
\end{equation*}
Further, by a direct calculation, one obtains:
\begin{equation*}
\begin{split}
\big(b^h\big)_{tan} & = 
\mbox{Id}_2 - x_3h^\alpha \nabla^2v_0 + o(h^{\alpha+1}),\\
\big(b^h\big)_{13, 23}^T & = h^\alpha \nabla v_0 + o(h^{\alpha+1}),\qquad
\big(b^h\big)_{13, 23, 33}  = \vec n^h.
\end{split}
\end{equation*}
and the inverse matrix $(b^h)^{-1}$ has the following structure:
\begin{equation*}
\begin{split}
\big((b^h)^{-1}\big)_{tan} & = A + o(h^{\alpha+1}),\\
\big((b^h)^{-1}\big)_{13, 23} & = h^\alpha A\nabla v_0 +
o(h^{\alpha+1}),\qquad  \big((b^h)^{-1}\big)^T_{13, 23, 33} = \vec n^h.
\end{split}
\end{equation*}
where the principal minor $A(x) \in\mathbb{R}^{2\times 2}$ of
$(b^h(x))^{-1}$ is the symmetric matrix:
\begin{equation}\label{A}
A = \big(\mbox{Id}_2 +h^{2\alpha}\nabla v_0\otimes\nabla v_0 -
x_3h^\alpha\nabla^2 v_0\big)^{-1}.
\end{equation}

{\bf 3.}  We now compute, 
\begin{equation*}
\begin{split}
\big((b^h)^{-1, T} M^h\big)_{tan}  & = \mbox{Id} + x_3h^\alpha A(\nabla^2v_0
- 2\nabla^2v)  + x_3h^\alpha \vec n^h_{tan} \otimes \vec d^h_{tan} +
o(h^{\alpha+1}) \\
& = \mbox{Id}_2 + x_3h^\alpha (\nabla^2v_0 - 2\nabla^2v) +
o(h^{\alpha+1}) \\
\big((b^h)^{-1, T} M^h\big)^T_{13,23}  & = h^\alpha(\mbox{Id} +
h^{2\alpha}\nabla v_0\otimes\nabla v_0) A \nabla v_0 + x_3h^\alpha
d^h_{tan} + o(h^{\alpha+1}) \\ & =  h^\alpha \nabla v_0   + x_3h^\alpha d^h_{tan} + o(h^{\alpha+1}) \\
\big((b^h)^{-1, T} M^h\big)_{13,23} & = x_3h^\alpha d^h_{tan} + \vec n^h_{tan} 
 + o(h^{\alpha+1}) \\
\big((b^h)^{-1, T} M^h\big)_{33} & = \vec n_3^h + 2x_3h^\alpha d_3^h   + o(h^{\alpha+1}),
\end{split}
\end{equation*}
where we used that $A\big(\mbox{Id} + h^{2\alpha}\nabla v_0\otimes
\nabla v_0 - 2x_3h^\alpha\nabla^2v\big) = \mbox{Id} + x_3h^\alpha
A(\nabla^2v_0 - 2\nabla^2v)$ and that
$h^\alpha\big(\mbox{Id} + h^{2\alpha}\nabla v_0\otimes\nabla v_0\big)A
= h^\alpha(\mbox{Id}+x_3h^\alpha\nabla^2 v_0A) = h^\alpha\mbox{Id} + o(h^{\alpha+1})$.
Consequently:
\begin{equation*}
\begin{split}
K^h_{tan}  & = A + x_3h^\alpha (\nabla^2v_0 - 2\nabla^2v) 
+ \vec n^h_{tan} \otimes \vec n^h_{tan} + o(h^{\alpha+1})  \\ & = A\big(\mbox{Id} +
x_3h^\alpha(\nabla^2v_0 - 2\nabla^2v) + \vec n^h_{tan} \otimes \vec
n^h_{tan}  + \frac{h^{4\alpha}|\nabla v_0|^2}{1+h^{2\alpha}|\nabla
  v_0|^2}\nabla v_0\otimes\nabla v_0\big) + o(h^{\alpha+1}) \\ & = 
A\big(\mbox{Id} +
x_3h^\alpha\nabla^2v_0 - 2x_3h^\alpha\nabla^2v + h^{2\alpha} \nabla v_0\otimes\nabla v_0\big) + o(h^{\alpha+1}) 
\\ & = \mbox{Id}_2 + 2x_3h^\alpha A(\nabla^2v_0 - \nabla^2v) + o(h^{\alpha+1})
\\ & = \mbox{Id}_2 + 2x_3h^\alpha(\nabla^2v_0 - \nabla^2v)  + o(h^{\alpha+1}) \\
K^h_{13,23}  & = h^\alpha A\nabla v_0 + x_3h^\alpha d^h_{tan} +\vec
n_3^h\vec n^h_{tan} + o(h^{\alpha+1}) \\ & = h^\alpha A\nabla v_0 + x_3h^\alpha d^h_{tan} 
- \frac{h^{\alpha}}{1+h^{2\alpha}|\nabla v_0|^2}\nabla v_0 + o(h^{\alpha+1}) \\
K^h_{33} & = h^{2\alpha}\langle A\nabla v_0, \nabla v_0\rangle +
2x_3h^\alpha d^h_3 +|\vec n_3^h|^2 + o(h^{\alpha+1}) \\ & = 
1+ h^{2\alpha}\langle A\nabla v_0, \nabla v_0\rangle +
2x_3h^\alpha d^h_3 -\frac{h^{2\alpha}|\nabla
  v_0|^2}{1+h^{2\alpha}|\nabla v_0|^2} + o(h^{\alpha+1}),
\end{split}
\end{equation*}
where we used that $ \vec n^h_{tan} \otimes \vec n^h_{tan} =
\frac{h^{2\alpha}}{1+h^{2\alpha}|\nabla v_0|^2}\nabla v_0\otimes\nabla
v_0$. Observe that:
\begin{equation*}
\begin{split}
& h^\alpha A\nabla v_0 -  \frac{h^{\alpha}}{1+h^{2\alpha}|\nabla
  v_0|^2}\nabla v_0 =  \\ & 
\qquad = \frac{h^{\alpha}}{1+h^{2\alpha}|\nabla
  v_0|^2} A \big( (1+h^{2\alpha}|\nabla v_0|^2)\nabla v_0 
- \nabla v_0 - h^{2\alpha} |\nabla v_0|^2 \nabla v_0 \big) +
o(h^{\alpha+1}) \\ & \qquad = o(h^{\alpha+1}).
\end{split}
\end{equation*}
Therefore, in fact:
\begin{equation*}
\begin{split}
K^h_{13,23}   = x_3h^\alpha d^h_{tan} + o(h^{\alpha+1}), \qquad 
K^h_{33}  = 1 + 2x_3h^\alpha d^h_3 + o(h^{\alpha+1}).
\end{split}
\end{equation*}
Concluding, we get:
$$K^h = \mbox{Id}_3 + 2x_3 h^\alpha\big( (\nabla^2v_0 - \nabla^2v)^* +
\mbox{sym}(d^h\otimes e_3)\big) + o(h^{\alpha+1}).$$

\medskip

{\bf 4.}  Taylor expanding $W$ at $\mbox{Id}_3$ and using (\ref{warp}) we now see that:
\begin{equation*}
\begin{split}
W(\sqrt{K^h}) & = W\big(\mbox{Id}_3 + x_3h^\alpha ((\nabla^2v_0 - \nabla^2v)^* +
\mbox{sym}(d^h\otimes e_3)) + o(h^{\alpha+1})\big) \\ & = 
\frac{1}{2} x_3^2h^{2\alpha}\mathcal{Q}_3 \big((\nabla^2v_0 - \nabla^2v)^* +
\mbox{sym}(d^h\otimes e_3)\big) + o(h^{2\alpha+2}).
\end{split}
\end{equation*}
Note that $\nabla\tilde \phi_h = 1 +\mathcal{O}(h^\alpha)$.
By this fact,  recalling (\ref{nowe}) and the convergence in (\ref{warp}),  it follows that:
\begin{equation*}
\begin{split}
\lim_{h\to 0}\frac{1}{h^{2\alpha+2}}I^h(u^h) &= \lim_{h\to
  0}\frac{1}{h^{2\alpha+2}} \frac{1}{h}\int_{\Omega^h} W(\sqrt{K^h})
(1+\mathcal{O}(h^\alpha)) \\ 
& =  \lim_{h\to  0}\frac{1}{2h^3} \int_{\Omega^h}
x_3^2\mathcal{Q}_3 \big((\nabla^2v_0 - \nabla^2v)^* +
\mbox{sym}(d\otimes e_3)\big) \\
& =  \lim_{h\to  0}\frac{1}{2h^3} \big(\int_{-h/2}^{h/2}
x_3^2~\mbox{d}x_3\big) \int_{\Omega}\mathcal{Q}_2\big(\nabla^2v_0 -
\nabla^2v\big)~\mbox{d}x \\ &
= \frac{1}{24}  \int_{\Omega}\mathcal{Q}_2\big(\nabla^2v_0 -
\nabla^2v\big)~\mbox{d}x.
\end{split}
\end{equation*}
Since, clearly $v^h_3 = h^\alpha v + \mathcal{O}(h^{2\alpha})$, we obtain:
\begin{equation*}
\begin{split}
\lim_{h\to 0}\frac{1}{h^{2\alpha+2}} \frac{1}{h}\int_{\Omega^h}
h^{\alpha+2} f v_3^h \det b^h 
= \lim_{h\to 0}\frac{1}{h^{2\alpha+2}} \frac{1}{h}\int_{\Omega^h}
h^{\alpha+2} f (h^\alpha v + \mathcal{O}(h^{2\alpha}))
= \int_\Omega fv~\mbox{d}x.
\end{split}
\end{equation*}
The proof of Theorem \ref{upbd} is complete.
\endproof
 
\bigskip

\noindent{\bf Acknowledgments.} 
This project is based upon work supported by, among others, the National Science
Foundation. M.L. is partially supported by the NSF grants DMS-0707275 and
DMS-0846996, and by the Polish MN grant N N201 547438. L.M.  
is supported by the MacArthur Foundation, M.R.P. is partially
supported by the NSF grants DMS-0907844 and DMS-1210258.

\end{document}